\documentclass[11pt]{article}

\usepackage{amsfonts,amsmath,latexsym}
\newtheorem{theorem}{Theorem}[section]
\newtheorem{lemma}{Lemma}[section]
\newtheorem{proposition}{Proposition}[section]
\newtheorem{conjecture}{Conjecture}[section]

\newcommand{\qed}{\hfill \ensuremath{\Box}}

\newenvironment{proof}
      {\noindent{\bf Proof:}\hspace{1mm}}
      {\hfill$\Box$\medskip}
\newenvironment{proofof}[1]
      {\noindent{\bf Proof of #1:}\hspace{1mm}}
      {\hfill$\Box$\medskip}
\input{epsf}

\makeatletter
\def\Ddots{\mathinner{\mkern1mu\raise\p@
\vbox{\kern7\p@\hbox{.}}\mkern2mu
\raise4\p@\hbox{.}\mkern2mu\raise7\p@\hbox{.}\mkern1mu}}
\makeatother

\def\BBE{\mathbb{E}}

\title{\vspace{-0.7cm} Chromatic number, clique subdivisions, and the conjectures of
Haj\'os and Erd\H{o}s-Fajtlowicz}
\author{Jacob Fox\thanks{Department of
Mathematics, MIT, Cambridge, MA 02139-4307. Email: {\tt
fox@math.mit.edu}. Research supported by a Simons Fellowship.}
\and
Choongbum Lee\thanks{Department of Mathematics, UCLA,
Los Angeles, CA, 90095. Email: {\tt choongbum.lee@gmail.com}.
Research supported in part by a Samsung Scholarship.}
\and
Benny Sudakov\thanks{Department of Mathematics,
UCLA,  Los Angeles, CA 90095. Email: {\tt bsudakov@math.ucla.edu}. Research
supported in part by NSF grant DMS-1101185, NSF CAREER award DMS-0812005 and by a
USA-Israeli BSF grant.}}
\date{}
\begin{document}
\maketitle

\begin{abstract}
For a graph $G$, let $\chi(G)$ denote its chromatic number and
$\sigma(G)$ denote the order of the largest clique subdivision in $G$. Let
$H(n)$ be the maximum of $\chi(G)/\sigma(G)$ over all $n$-vertex
graphs $G$. A famous conjecture of Haj\'os from 1961 states that
$\sigma(G) \geq \chi(G)$ for every graph $G$. That is, $H(n) \leq 1$
for all positive integers $n$. This conjecture was disproved by
Catlin in 1979. Erd\H{o}s and Fajtlowicz further showed
by considering a random graph that $H(n) \geq cn^{1/2}/\log n$ for
some absolute constant $c>0$. In 1981 they conjectured that this bound
is tight up to a constant factor in that there is some absolute constant $C$ such that
$\chi(G)/\sigma(G) \leq Cn^{1/2}/\log n$ for
all $n$-vertex graphs $G$. In this paper we prove the Erd\H{o}s-Fajtlowicz
conjecture. The main ingredient in our proof, which might be of independent interest, 
is an estimate on the order of the largest clique subdivision which one can find in every graph on $n$ vertices with
independence number $\alpha$.

\end{abstract}

\section{Introduction}
A {\em  subdivision} of a graph $H$ is any graph formed by replacing edges of $H$ by internally vertex disjoint paths.
This is an important notion in graph theory, e.g., the celebrated theorem of Kuratowski uses it to
characterize planar graphs. For a graph $G$, we let $\sigma(G)$ denote the largest integer $p$ such that $G$ 
contains a subdivision of a complete graph of order $p$.
Clique subdivisions in graphs have been extensively studied and there are many results which give 
sufficient conditions for a graph $G$ to have large $\sigma(G)$.
For example, Bollob\'as and Thomason \cite{BT},
and Koml\'os and Szemer\'edi \cite{KS} independently proved that
every graph of average degree at least $d$ has $\sigma(G) \ge c d^{1/2}$
for some absolute constant $c$. Motivated by a conjecture of Erd\H{o}s, in \cite{AKS} the authors further showed that
when $d= \Omega(n)$ in the above subdivision one can choose all paths to have length two.
Similar result for subdivisions of general graphs with $O(n)$ edges (a clique of order $O(\sqrt{n})$ clearly satisfies this)
was obtained in \cite{FS0}.

For a given graph $G$, let $\chi(G)$ denote 
its chromatic number. A famous conjecture made by Haj\'os 
in 1961 states that $\sigma(G) \ge \chi (G)$. 
Dirac \cite{Di} proved that this conjecture is true for all $\chi(G) \le 4$,
but in 1979, Catlin \cite{Ca} disproved the conjecture for all $\chi(G) \ge 7$. 
Subsequently, several researchers further studied this problem. 
On the negative side, by considering random graphs, Erd\H{o}s and Fajtlowicz \cite{EF} in 1981 showed that 
the conjecture actually fails for almost all graphs. On
the positive side, recently K\"uhn and Osthus \cite{KO}
proved that all graphs of girth at least 186 satisfy Haj\'os' conjecture.
Thomassen \cite{Th} studied the relation of Haj\'os' conjecture
to several other problems of graph theory such as
Ramsey theory, maximum cut problem, etc., and discovered many interesting
connections.

In this paper, we revisit Haj\'os' conjecture and study to what extent
the chromatic number of a graph can exceed the order of its largest clique subdivision.
Let $H(n)$ denote the maximum of $\chi(G)/\sigma(G)$ over all $n$-vertex graphs $G$.
The example of graphs given by Erd\H{o}s and Fajtlowicz 
which disprove Haj\'os' conjecture in fact has
$\sigma(G) = \Theta(n^{1/2})$ and $\chi(G) = \Theta(n/\log n)$. Thus 
it implies that $H(n) = \Omega(n^{1/2}/\log n)$. 
In \cite{EF}, Erd\H{o}s and Fajtlowicz conjectured that this bound is tight up to a constant factor so that $H(n) = O(n^{1/2}/\log n)$. Our first theorem verifies this conjecture.
 
\begin{theorem}\label{main}
There exists an absolute constant $C$ such that $H(n) \leq Cn^{1/2}/\log n$ for $n \geq 2$.
\end{theorem}

The proof shows that we may take $C=10^{120}$, although we do not try to optimize this constant. For the random graph $G=G(n,p)$ with $0<p<1$ fixed, Bollob\'as and Catlin \cite{BC} determined $\sigma(G)$ asymptotically almost surely and later Bollob\'as \cite{Bo} determined $\chi(G)$ asymptotically almost surely. These results imply, by picking the optimal choice $p=1-e^{-2}$, the lower bound $H(n) \geq (\frac{1}{e\sqrt{2}}-o(1))n^{1/2}/\log n$. 

 For a graph $G$, let $\alpha(G)$ denote its independence number. Theorem \ref{main} actually follows from the study of 
the relation between $\sigma(G)$ and $\alpha(G)$, which might be of independent interest.
Let $f(n,\alpha)$ be the minimum of $\sigma(G)$ over all graphs $G$
on $n$ vertices with $\alpha(G) \le \alpha$. 

\begin{theorem}\label{main1}
There exist absolute positive constants $c_1$ and $c_2$ such that the following holds.
\begin{enumerate}
	\setlength{\itemsep}{1pt}
  \setlength{\parskip}{0pt}
  \setlength{\parsep}{0pt}
\item If $\alpha<2\log n$, then $f(n,\alpha) \ge c_1 n^{\frac{\alpha}{2\alpha-1}}$, and
\item if $\alpha = a\log n$ for some $a \ge 2$, then $f(n,\alpha) \ge c_2 \sqrt{\frac{n}{a\log a}}$.
\end{enumerate}
\end{theorem}

Note that for $\alpha = 2\log n$, both bounds from the first and
second part gives $f(n,\alpha) \ge \Omega(\sqrt{n})$.  
Moreover, both parts of this theorem establish the
correct order of magnitude of $f(n,\alpha)$ for some 
range of $\alpha$. For $\alpha=2$, it can be shown that in
the triangle-free graph constructed by Alon \cite{Al},
every set of size at least $37 n^{2/3}$ contains at least $n$
edges. This implies that the complement of this graph 
has independence number 2 and the largest clique subdivision of
size $t<37n^{2/3}$. Indeed, if there is a clique subdivision of order $t \geq 37n^{2/3}$, then between each of the at least $n$ 
pairs of nonadjacent vertices among the $t$ vertices of the subdivided clique, there is at least one additional vertex along 
the path between them in the subdivision. However, this would require at least $t+n$ vertices in the $n$-vertex graph, a contradiction. On the other hand, for $\alpha = \Theta(\log n)$, 
by considering $G(n,p)$ with constant $0<p<1$, one can see that
the second part of Theorem \ref{main1} is tight up to the constant factor. Even for $\alpha=o(\log n)$, by considering
the complement of $G(n,p)$ for suitable $p \ll 1$, one can easily verify that
there exists an absolute constant $c'$ such that $f(n, \alpha) \le O(n^{\frac{1}{2} + \frac{c'}{\alpha}})$.

Theorem \ref{main1} can also be viewed as a Ramsey-type theorem which establishes an upper bound on the Ramsey number of a clique subdivision versus an independent set.

\medskip 

\noindent \textbf{Notation}. A graph $G=(V,E)$ is given by a pair
of its vertex set $V=V(G)$ and edge set $E=E(G)$. The edge density of $G$ is the ratio $|E|/{|V| \choose 2}$.
For a subset $X$ of vertices,
we use $G[X]$ to denote the induced subgraph of $G$ on
the set $X$. Throughout the
paper log denotes the natural logarithm. We systematically omit floor and ceiling signs whenever they are not crucial, for the sake of clarity of presentation. We also do not make any serious attempt to optimize absolute constants in our statements and proofs.

\section{Deducing Theorem \ref{main} from Theorem \ref{main1}}

Theorem \ref{main} is a quick corollary of Theorem \ref{main1}. 
To see this, let $C=\max\left(e^8,\frac{16}{c_1e},\frac{4}{c_2\sqrt{e}}\right)$, where $c_1, c_2$ are the constants from Theorem \ref{main1}. We will prove Theorem \ref{main} by induction on $n$.
Suppose we want to prove the claim for $n$, and
are given a graph $G$ on $n$ vertices with $\chi(G)=k$. We may assume that
$k \ge C$ (and thus $n \ge C$) as otherwise the claim is trivially true.
If $\alpha:=\alpha(G) < 4n/k$, then the bounds in Theorem \ref{main1} easily 
give us the desired bound. Indeed, consider the two cases. Let $a=\alpha/\log n$. If also 
$\alpha < 2 \log n$, so that $a<2$, 
we can use the first part of Theorem \ref{main1} to get
\[ \frac{\chi(G)}{\sigma(G)} \le \frac{k}{c_1 n^{\frac{\alpha}{2\alpha - 1}}} \le \frac{k}{c_1 n^{\frac{1}{2} + \frac{1}{4\alpha}}} 
< \frac{4n/\alpha}{c_1 n^{\frac{1}{2} + \frac{1}{4\alpha}}} 
=\frac{n^{1/2}}{\log n} \cdot \frac{4}{c_1 a e^{\frac{1}{4a}}}\leq \frac{n^{1/2}}{\log n} \cdot \frac{16}{c_1 e}
 \le C \frac{n^{1/2}}{\log n}, \]
where we used the fact that the minimum of $a e^{\frac{1}{4a}}$ in the domain $(0,2]$ occurs at $a=1/4$. 
If $\alpha(G) \geq 2 \log n$, so that $a \geq 2$, then 
by using the second part of Theorem \ref{main1} we get 
\[ \frac{\chi(G)}{\sigma(G)} \le \frac{\sqrt{a \log a} \cdot k}{c_2 \sqrt{n}} \le 
\frac{\sqrt{a \log a} \cdot 4n/(a\log n)}{c_2 \sqrt{n}} =
\frac{n^{1/2}}{\log n} \cdot \frac{4\sqrt{ \log a}}{c_2\sqrt{a}} \leq \frac{n^{1/2}}{\log n} \cdot \frac{4}{c_2 \sqrt{e}} \le C \frac{n^{1/2}}{\log n}, \]
where we used the fact that the maximum of $\frac{\log a}{a}$ occurs at $a=e$. 

Otherwise, $\alpha(G) \geq 4n/k$. By deleting a maximum independent
set, we get an induced subgraph $G'$ on $n' \leq n-4n/k$ vertices,
with chromatic number at least $k-1$, and a
clique subdivision of size at least $\sigma(G') \geq
\chi(G')/H(n') \geq (k-1)/H(n')$. 
Note that if $n' < e^2$, then $k \leq 1+n' < 9<C$, and this case 
was already settled. So we may assume $n' \geq e^2$. 
Hence, by induction on $n$, we have 
\begin{equation}\label{123abc}\frac{\chi(G)}{\sigma(G)} \leq
\frac{k}{\sigma(G')} \leq \left(\frac{k}{k-1}\right)H(n') \leq
\frac{k}{k-1} \cdot \frac{Cn'^{1/2}}{\log n'}. \end{equation} 
As the function $x^{1/2}/\log x$ is increasing for $x \geq e^2$, the right-hand side of (\ref{123abc}) is
maximized at $n' = (1 - 4/k)n$. Consequently, we have
\begin{eqnarray*} \frac{\chi(G)}{\sigma(G)}  & \leq & 
\frac{k}{k-1} \cdot \frac{Cn'^{1/2}}{\log n'} \leq \left(\frac{k}{k-1}\right)\left(1-\frac{4}{k}\right)^{1/2}\left(1+\frac{\log(1-4/k)}{\log n}\right)^{-1}\frac{Cn^{1/2}}{\log n} \\ & \leq & \left(1-\frac{1}{k}\right)\left(1+\frac{\log(1-4/k)}{\log n}\right)^{-1}\frac{Cn^{1/2}}{\log n} \leq \left(1-\frac{1}{k}\right)\left(1-\frac{8}{k\log n}\right)^{-1}\frac{Cn^{1/2}}{\log n}  \\ & \leq &  C\frac{n^{1/2}}{\log n},  \end{eqnarray*}
where in the second to last inequality we used $$\log(1-x)=-(x+\frac{x^2}{2}+\frac{x^3}{3}+\cdots)>-(x+x^2+x^3+\cdots)>-2x$$ for $0<x<1/2$, which holds with $x=4/k$ from $k \geq C \geq e^8$, and in the last inequality we used $n \geq C \geq e^8$. We therefore have $H(n) \leq Cn^{1/2}/\log n$, which completes the proof. \qed

\section{Tools and the idea of the proof}
\label{sec:tools}

The proof of Theorem \ref{main1}
makes use of four main tools that we describe in this section. In the end of the section we outline the proof of Theorem \ref{main1}
using these tools.

Our first tool is a theorem independently proved by Bollob\'as and Thomason
\cite{BT}, and Koml\'os and Szemer\'edi \cite{KS}. They 
determined up to a constant factor the
minimum number of edges which guarantees a $K_t$-subdivision in a graph on $n$
vertices, solving  an old
conjecture made by Erd\H{o}s and Hajnal, and also by Mader.

\begin{theorem}\label{BTKS} {\bf (Bollob\'as-Thomason, Komlos-Szemer\'edi)}
Every graph $G$ with $n$ vertices and at least $256t^2 n$ edges
satisfies $\sigma(G) \geq t$.
\end{theorem}

We remark that Theorem \ref{BTKS} implies $H(n)=O(n^{1/2})$. Indeed,
a graph $G$ with chromatic number $k$ has a subgraph with minimum
degree at least $k-1$ and hence by Theorem \ref{BTKS} satisfies
$\sigma(G) =\Omega(k^{1/2})$. We thus get $\frac{\chi(G)}{\sigma(G)}
= O(k^{1/2})=O(n^{1/2})$. As this bound holds for all $G$ on $n$
vertices, we have $H(n)=O(n^{1/2})$.
Our goal is to prove the better bound $H(n)=O(n^{1/2}/\log n)$.

\medskip

The theorem above can be used as a black box to indirectly construct
clique subdivisions in certain cases. However, in order to directly
construct a large clique subdivision in a graph, we first
find a large subset in which only a small number of edges is
missing, and then for each such missing edge, find internally vertex-disjoint paths connecting the two endpoints. For technical reasons, we reverse the two steps. That is, we first find a large
subset of vertices such that every pair of vertices can be connected
by many internally vertex-disjoint paths (Lemma \ref{DRC}), and then find a
further subset in which only a small proportion of edges is missing
(Lemma \ref{snlem}). It would be nice if these two steps were sufficient in proving Theorem \ref{main}.
Unfortunately, a naive application of these two
steps together with Theorem \ref{BTKS} will only imply our main
result for a certain range of parameters, more precisely, when the
graph is dense enough depending on the independence number. Thus to
handle the case when the graph is sparse, we develop another lemma
(Lemma \ref{ES}), which essentially says that in the sparse case,
we can find a subgraph in which the parameters work (see the discussion
before Lemma \ref{ES}).

The next tool is based on a simple yet surprisingly powerful lemma whose proof uses a probabilistic argument known as dependent random choice. Early versions of this technique were developed in the papers \cite{Go,KoRo,Su1}. Later, variants were discovered and applied to various problems in Ramsey theory and extremal graph theory (see the survey \cite{FS} for more details). The following lemma says that every graph of large enough density contains a large subset in which every pair of vertices are connected by many internally vertex-disjoint paths.

\begin{lemma}\label{DRC}
Assume that $d$ and $n$ are given so that $d^2 n \ge 1600$. If
$G=(V,E)$ has $n$ vertices and edge density $d$, then there is a
vertex subset $U \subset V$ with $|U| \ge dn/ 50$ vertices such that
every pair of vertices in $U$ have at least $10^{-9} d^{5}n$
internally vertex-disjoint paths of length $4$ which uses only
vertices from $V \setminus U$ as internal vertices.
\end{lemma}
\begin{proof}
Let $V_1$ be a random subset of $V$ of size $\lceil n/2 \rceil$
and let $V_2=V \setminus V_1$. Then it is easy to check that 
$\BBE[e(V_1, V_2)] \ge \frac{d}{2} {n \choose 2}$ and therefore we may pick such a partition $V=V_1 \cup V_2$ with $e(V_1, V_2) 
\ge 
\frac{d}{2} {n \choose 2}$. Throughout the proof we restrict
our graph to the bipartite graph induced by the edges between $V_1$
and $V_2$.

Pick a vertex $v_0 \in V_2$ uniformly at random,
and let $X \subset V_1$ be the neighborhood of $v_0$. The
probability of a fixed vertex $v \in V_1$ belonging to $X$ is
$\mathbb{P}(v \in X) = \deg(v)/|V_2|$ and thus by the Cauchy-Schwarz
inequality and $n\ge 2$,
\[ \BBE[|X|^2] \ge \BBE[|X|]^2 = \left( \sum_{v \in V_1} \frac{ \deg(v) }{|V_2|} \right)^2
  = \left(\frac{e(V_1, V_2)}{|V_2|}\right)^2 \ge \left(\frac{d(n-1)}{2}\right)^{2} \ge \frac{d^2n^2}{16}. \]
Call a pair $(v,w)$ of vertices in $V_1$ {\em bad} if $v$ and $w$ have at most $d^2 n / 800$ common neighbors, and call it {\em good} otherwise. Let $b$ be the number of bad pairs in $X$. Note that for
a bad pair $(v,w)$, the probability that both $v$ and $w$ belongs to
$X$ is at most $\mathbb{P}(v,w \in X) \le d^2 n/(800|V_2|) $.
Consequently, the expectation of $b$ is at most
\[ \BBE[b] \le \frac{d^2 n}{800|V_2|} {n \choose 2} \le \frac{d^2 n^2}{800}. \]
Thus we have
\[ \BBE[|X|^2 - 40 b] \ge \frac{d^2 n^2}{80}. \]
Therefore we have a choice of $v_0$ for which $|X|^2 - 40b \ge
d^2n^2 / 80 \ge 0$. Fix this choice of $v_0$ (and $X$). Note that
this in particular implies $|X| \ge dn / 10$ and $b \le |X|^2/40$.

Call a vertex in $X$ {\em bad} if it forms a bad pair with at least
$|X|/4$ vertices in $X$. By the bound on $b$, we know that there
are at most $|X|/5$ bad vertices in $X$. Let $U$ be an arbitrary
subcollection of non-bad vertices in $X$ of size $|X|/5 \ge dn/50$.
We claim that $U$ is a set which has all the claimed properties.

Since the vertices of $U$ form a bad pair with at most $|X|/4$
vertices in $X$, for every two distinct vertices $v,w$ in $U$, the
number of vertices in $X \setminus U$ with which both $v,w$ form a good pair is at least
\[ |X| - |U| - 2 \cdot \frac{|X|}{4} = \frac{3|X|}{10}. \]
Moreover, whenever we have a vertex $x$ which forms a good pair with
both $v$ and $w$, by the definition of a good pair, we can find at
least $(d^2n/800)(d^2n/800 - 1)$ paths of length 4 connecting $v$
and $w$ which uses only vertices from $V \setminus U$ as internal
vertices. Therefore by collecting the facts, we see that given $d^2n
\ge 1600$, the number of such paths of length 4 between $v$ and $w$
is at least
\[ \left( \frac{3|X|}{10} \right) \cdot \left(\frac{d^2 n}{800} \right)\left(\frac{d^2 n}{800} - 1\right)
\ge \frac{3d^5 n^3}{100 \cdot 800 \cdot 1600} \ge 10^{-8}d^5 n^3. \]

Note that interior (without endpoints) of any given path of length 4 connecting  $v$ and $w$ can 
intersect at most $3n^2$ other such paths. This implies that 
there are at least $10^{-8}d^5 n^3/(3n^2)\geq 10^{-9}d^5 n$ internally  vertex-disjoint paths
connecting $v$ and $w$ which uses only vertices from $V \setminus U$
as internal vertices. This completes the proof. 
\end{proof}

The following lemma asserts that every graph of small independence number contains a large subset in which only a small proportion of edges are missing.

\begin{lemma}\label{snlem}
Let $0<\rho < 1$ and $\alpha$ be a positive integer. 
Then for every positive
integer $s \le \lceil \rho^{\alpha -1} n \rceil$, every graph $G$ on
$n$ vertices with independence number at most $\alpha$ contains a
subset of size $s$ with at most $\rho s^2$ nonadjacent pairs of
vertices.
\end{lemma}
\begin{proof}
If $s=1$, then the claim is clearly true. Thus we assume that $s \ge
2$. Let $t$ be an integer satisfying $t \ge s$. It suffices to find
a subset of order $t$ which has at most $\rho t^2 / 2$ nonadjacent
pairs, since by an averaging argument over all subsets of this $t$-set of order $s$, we can find a subset of order
$s$ which has at most
\[ \frac{\rho t^2}{2 {t \choose 2}} \cdot {s \choose 2} \le \rho s^2 \]
edges missing.

Let $V_0=V(G)$. We will find a sequence $V_0 \supset V_1 \supset \cdots$ of subsets such that the induced subgraph of $G$ with vertex set $V_i$ has independence number at most $\alpha-i$ and at least $\rho^{i}n$ vertices. Notice this is satisfied for $i=0$. If $V_i$ has a vertex which has at least $\rho|V_i|$ non-neighbors  in $V_i$, then let $V_{i+1} \subset V_i$ be the subset of non-neighbors, so $|V_{i+1}| \geq \rho|V_i|$. Since the induced subgraph of $G$ with vertex set $V_{i+1}$ has independence number at most $\alpha-i-1$, we can continue the induction. Otherwise, every vertex of $V_i$ has less than $\rho|V_i|$ non-neighbors, so there are less than $\rho |V_i|^2/2$ nonadjacent pairs in $V_i$, in which case we are done. If this process continues through $\alpha-1$ steps, we get a set $V_{\alpha-1}$ of order at least $\rho^{\alpha-1}n$, and independence number at most one, so this is a clique of order at least $s$, which completes the proof.
\end{proof}


Suppose we are trying to prove Theorem \ref{main1} for $\alpha < 2\log n$.
First apply Lemma \ref{BTKS} to find a subset $U$ of size $\Omega(dn)$ in which
each pair is connected by $\Omega(d^5n)$ internally vertex-disjoint paths of length $4$. Then apply
Lemma \ref{snlem} to $U$ with a suitable choice of $\rho$, and hope to find
a subset of size $\Omega(n^{\alpha/(2\alpha-1)})$, in which $O(d^5n)$ edges are missing.
By Lemma \ref{DRC} we can use internally vertex-disjoint paths of length 4 instead of missing edges to get 
a clique subdivision on these vertices. A crucial observation is that this only works if $U$ is large enough
(that is, if $d$ is large enough). 

Our next lemma can be used to
overcome this difficulty. The idea of this lemma first appeared in a
1972 paper of Erd\H{o}s-Szemer\'edi \cite{ES}, and has also been
useful in other problems (for example, \cite{Su}). It shows that if
a sparse graph does not have large independence number, then it contains an induced subgraph with many vertices whose independence number is much smaller then in the original graph. We will later see that with the help of this lemma, the strategy above can be modified to find a subset of
size $\Omega(n^{\alpha/(2\alpha-1)})$ with $O(d^5 n)$ non-adjacent 
pairs (we use the same strategy for $\alpha \ge 2\log n$).


\begin{lemma}\label{ES}
Let $0 < d \le 1$. Let $G$ be a graph, $I$ be a maximum
independent set of $G$ with $|I|=\alpha$, and $V_1$ be a vertex
subset of $V \setminus I$ with $|V_1|=N$ such that each vertex in $V_1$ has at most $d|I|$
neighbors in $I$. Then there is a subset $U \subset V_1$ with $|U| \geq
\left(\frac{e}{d}\right)^{-d\alpha}N$ such that the induced subgraph
of $G$ with vertex set $U$ has independence number at most
$d\alpha$.
\end{lemma}
\begin{proof}
For every vertex $v \in V_1$ fix a subset of $I$ of size $\lfloor d|I| \rfloor$ which contains 
all neighbors of $v$ in $I$. Since the number of such subsets of $I$ is at most
${|I| \choose \lfloor d|I| \rfloor} \leq \left(\frac{e}{d}\right)^{d \alpha}$, one of them contains neighborhoods of at least
$|V_1|\left(\frac{e}{d}\right)^{-d \alpha}$ vertices of $V_1$. 
Let $U \subset V_1$ be 
the set of these vertices. We have that $|U| \geq
\left(\frac{e}{d}\right)^{-d\alpha}N$, and the number of
vertices in $I$ which have a neighbor in $U$ is at most $d|I|$.

Note that all the vertices in $I$ which do not have a neighbor in $U$ can
be added to an independent subset of $U$ to make a larger independent set. Since $I$ is an independent set of size $\alpha$, there are at least $(1-d)\alpha$ such vertices. Moreover, since $G$ has independence number at most $\alpha$, the induced subgraph of $G$ on $U$ has independence number at most $d\alpha$.
\end{proof}

\subsection{Outline of the proof}

We next outline the proof of Theorem \ref{main1}, which gives a lower
bound on $\sigma(G)$ for a graph $G$ with $n$ vertices and independence number
$\alpha$. The proof strategy depends on whether or not the graph $G$
is dense.

When $G$ is dense the proof splits into two cases, depending on the size of  $\alpha$ 
(see Lemma \ref{dense} in the next section). If $\alpha \leq 2 \log n$, then we apply Lemma \ref{DRC} to
obtain a large vertex subset $U$ in which every pair of vertices in $U$ are
the endpoints of a large number of internally vertex-disjoint paths of
length $4$. We then apply Lemma \ref{snlem} to obtain a subset $S \subset
U$ of large order $s$ such that $G[S]$ has few missing edges. The vertices
of $S$ form the vertices of a $K_{s}$-subdivision. Indeed, for every pair
of adjacent vertices in $S$, we
use the edges between them as paths, and for every pair of non-adjacent
vertices, we use paths of length $4$ between them. These paths can be
chosen greedily using that each pair of vertices in $S$ are the endpoints
of many internally vertex-disjoint paths of length $4$ and there are few
missing edges within $S$.
This completes the case $\alpha \leq 2\log n$ of Lemma \ref{dense}. If
$\alpha>
2\log n$, using the fact $G$ is dense, we apply Theorem \ref{BTKS} to obtain
the desired large clique subdivision.

For sparse $G$ we prove a lower bound on $\sigma(G)$  in terms of the number
$n$ of vertices, the independence number $\alpha=\alpha(G)$, and the edge
density $d$ of $G$ (see Lemma \ref{sparse} in the next section). The proof is by induction on $n$,
the base case $n=1$ being trivial. The cases $d<n^{-1/4}$ or $\alpha \geq
n/16$ can be
trivially verified, so we may suppose $d>n^{-1/4}$ and $\alpha<n/16$. One
easily finds
an independent set $I$ and a vertex subset $V''$ which is disjoint 
from $I$ with $|V''| \geq n/8$ such that $I$ is a maximum independent 
set in $G[V'' \cup I]$ such that every vertex in $V''$ has at most $8d|I|$
neighbors in $I$. If $G[V'']$ has edge density at most $d/10$, then, by the
induction hypothesis, $G[V'']$, and hence $G$ as well,
contains a $K_s$-subdivision of the desired size. So we may suppose
$G[V'']$ has edge density at least $d/10$. Apply Lemma \ref{DRC} to find a
large subset $V_1 \subset V''$ such that
every pair of vertices in $V_1$ are the endpoints of a large number of
internally vertex-disjoint paths of length $4$. Applying Lemma \ref{ES},
we find a large
subset $U \subset V_1$ such that the independence number of $G[V_1]$ is
small. Finally, we then apply Lemma \ref{snlem} to obtain a subset $S
\subset U$ of large order $s$ such that $G[S]$ has few
missing edges. Just as in the dense case discussed above, the set $S$ form
the vertices of the desired $K_{s}$-subdivision.

\section{Proof of Theorem \ref{main1}}

In this section we prove Theorem \ref{main1} using the tools and
the strategy we developed in the previous section.
We separately consider two cases depending on the relation between
the edge density $d$ and the independence number $\alpha$ of the
graph. The following lemma establishes the case when the
graph is dense.

\begin{lemma}\label{dense}
Fix a constant $0 < c \leq 1$.
The following holds for every graph $G$ with $n \ge 10^{14}c^{-5}$
vertices, edge density $d$, and independence number $\alpha$.
\begin{enumerate}
\item[(i)] If $\alpha \le 2\log n$ and $d \ge c$,
then $\sigma(G) \ge 10^{-6} c^{5/2}
n^{\alpha/(2\alpha-1)}$.
\item[(ii)] If $\alpha = a \log n$ for some $a \ge 2$ and $d \ge c/(a \log a)$, then
$\sigma(G) \ge \sqrt{\frac{c}{600}} \sqrt{\frac{n}{a\log a}}$.
\end{enumerate}
\end{lemma}
\begin{proof}
(i) Given a graph $G$ as in the statement of the lemma, 
since $d^2 n \ge 1600$, we can apply
Lemma~\ref{DRC} to get a vertex subset $U$ of size $dn/50$ such that
every pair of vertices in $U$ have at least $10^{-9}d^5 n$
internally vertex-disjoint paths of length $4$ between them whose internal vertices lie in $V \setminus U$.

We may assume $\alpha \geq 2$, as otherwise $\alpha=1$, $G$ is a clique, and $\sigma(G)=n$. By applying Lemma~\ref{snlem} to $U$ with $\rho =
\left(10^{-7}d^3n^{-1}\right)^{1/(2\alpha-1)}$ (note that $\rho
< 1$), we find a vertex subset $S \subset U$ of size
\[ s = \lceil \rho^{\alpha-1}|U| \rceil = \left\lceil (10^{-7}d^3)^{(\alpha-1)/(2\alpha-1)} \cdot \frac{d}{50} \cdot n^{\alpha/(2\alpha-1)} \right\rceil \geq 10^{-6}d^{5/2}n^{\alpha/(2\alpha-1)}\]
with at most
\[ \rho s^2 \le 2\rho^{2\alpha-1}|U|^2 =2 \cdot 10^{-7}d^3n^{-1} \cdot \left(dn/50\right)^2
  \le 10^{-10} d^5 n, \]
nonadjacent pairs, 
where we used the fact that $s^2 \le 2\rho^{2\alpha-2}|U|^2$, which follows 
from the inequality $s \ge 10^{-6}d^{5/2}n^{\alpha/(2\alpha-1)} \ge 10^{-6}d^{5/2}n^{1/2} \ge 10$
(recall that $n \ge 10^{14}c^{-5}$).

We claim that the vertices of $S$ form the vertices of a
$K_s$-subdivision. For every pair of adjacent vertices in $S$, we
use the edges between them as paths, and for every pair of
non-adjacent vertices, we use paths of length $4$ between them.
Since the number of non-adjacent pairs of vertices is at most $10^{-10} d^5 n$, and each such pair has at least $10^{-9} d^5n \ge 3 \cdot 10^{-10} d^5 n$ internally vertex-disjoint paths
of length $4$ connecting them which uses only vertices from $V \setminus S$ as
internal vertices, we can greedily pick one path for each
non-adjacent pair to construct a $K_s$-subdivision. Indeed, 
note that the use of a certain path of length $4$ can destroy at most 3 other such paths for each other non-adjacent pair since they have disjoint interiors.

\noindent (ii) Since $d \ge c/(a \log a)$, the total number
of edges in the graph is
\[ d {n \choose 2} \ge \frac{c(n-1)}{2a \log a} n. \]
Therefore by Theorem~\ref{BTKS}, we can find a $K_s$-subdivision
for $s$ satisfying
\[ s \ge \sqrt{\frac{c(n-1)}{512 a \log a}} \ge \sqrt{\frac{c}{600}} \sqrt{\frac{n}{a \log a}}. \]
\end{proof}

Let $f(n,\alpha,d)$ be the maximum $t$ such that every graph $G$
on $n$ vertices with independence number at most $\alpha$ and edge
density $d$ contains a $K_t$-subdivision. First note that
by Tur\'an's theorem, we have a lower 
bound on $d$ in terms of $\alpha$.

\begin{proposition} \label{TURAN}
Let $G$ be a graph with $n$ vertices, edge density $d$, and independence number $\alpha$.
If $\alpha \le n/2$, then $d \ge 1/(2\alpha)$.
\end{proposition}
\begin{proof}
By Tur\'an's theorem and convexity of the function $g(x)={x \choose 2}$, we know
$d {n \choose 2} \ge \alpha {n/\alpha \choose 2} = n(n/\alpha-1)/2$,
from which it follows that
$d \ge \frac{(n/\alpha-1)}{n-1} \ge \frac{1}{2\alpha}$.
\end{proof}

The next lemma
establishes a good bound on $f(n,\alpha,d)$ when $d$ is small
(depending on $\alpha$). This can be used to handle the remaining
case of the proof of Theorem \ref{main1}.

\begin{lemma} \label{sparse}
Let $G$ be a graph with $n$ vertices, edge density $d$, and
independence number $\alpha$.  If $\alpha \le n/2$, $d \le 10^{-20}$, 
and $d\alpha \log(1/d) \le (\log n) /100$, then we have 
$f(n, \alpha, d) \ge (1/50)d^4n^{\frac{1}{2}+\frac{1}{40d\alpha}}$.
\end{lemma}
\begin{proof}
The proof is by induction on $n$. Note that by $\alpha \le n/2$ and
Proposition~\ref{TURAN}, we have $d \ge 1/(2\alpha)$. Thus we
only need to consider the range $1/(2\alpha) \le d \le 1$.

We first verify some initial cases. Namely, we prove that 
if either $d < n^{-1/4}$ or $\alpha > n/16$ holds, then the bound
$f(n, \alpha, d) \ge (1/50)d^4n^{\frac{1}{2}+\frac{1}{40d\alpha}}$ is true.

If $d < n^{-1/4}$, then by the fact $d \ge 1/(2\alpha)$ we have
\[ (1/50)d^{4} n^{\frac{1}{2} + \frac{1}{40d\alpha}} \le n^{-1} n^{\frac{1}{2} + \frac{1}{20}} \le 1. \]
Therefore the statement is true in this case by the trivial bound $f(n,\alpha,d) \ge 1$.
From now on, we may assume that $d\ge n^{-1/4}$, from which
it follows that $n \ge d^{-4} \ge 10^{80}$.

If $\alpha > n/16$ and $d \ge n^{-1/4}$, then by applying Theorem~\ref{BTKS} we can
find a clique subdivision of order at least
\[ \sqrt{d{n \choose 2}/{256n}} = \sqrt{\frac{d(n-1)}{512}} \ge \sqrt{\frac{dn}{600}}. \]
We claim that this is larger than $(1/50)d^4n^{\frac{1}{2}+\frac{1}{40d\alpha}}$. Indeed, it suffices to show that
$(1/50) d^4n^{\frac{1}{40d\alpha}} \le \sqrt{d/600}$,
which is implied by 
\[  d^{7/2}n^{\frac{1}{40d\alpha}} \le 2. \]
In the range $n^{-1/4} \le d \le 1$ with $\alpha > n/16$, the left-hand side is an increasing function and hence maximized at $d=1$. When $d=1$ we have
$d^{7/2} n^{\frac{1}{40d\alpha}} \le n^{\frac{2}{5n}} < 2$. This finishes the proof of the initial cases.

Now assume that some $n$ is given and the lemma has been
proved for all smaller values of $n$. By the observations above,
we may also assume that $d \ge 1/(2\alpha)$, $\alpha \le n/16$, and $d \ge n^{-1/4}$, which implies $n \ge 10^{80}$. 
Let $G$ be a graph on $n$ vertices with edge density $d$,
and independence number at most $\alpha$. Let $V' \subset V(G)$ be
the set of vertices of degree at most $2dn$, so $|V'| \geq n/2$.

Let $I$ be a maximum independent set in the induced subgraph of $G$
with vertex set $V'$. Note that $|I| \leq \alpha$. Let $X \subset V'$ be
the set of vertices with at least $8d|I|$ neighbors in $I$, and 
let $V'' = V' \setminus (X \cup I)$.
Then by counting the number of edges incident to vertices of $I$
in two different ways, we get
\[ |X| \cdot 8d|I| \le |I| \cdot 2dn, \]
from which we get the bound $|X| \le n/4$.
Therefore, we get $|V''| \geq |V'| - |X|  -|I|\geq n/8$.
Let $n' = |V''| \ge n/8$, and let $d'$ be the edge density and 
$\alpha'$ be the independence number, respectively, of the graph $G[V'']$.
Note that $\alpha' \le \alpha$. 

\noindent {\bf Case 1} : $d' \le d/10$.

We want to apply the inductive hypothesis in this case. 
We first show that the new parameters $n', d', \alpha'$ satisfy
all the imposed conditions.
First we have $\alpha' \le \alpha \le n/16 \le n'/2$, and second
$d' \le d \le 10^{-20}$. Finally, since $t\log (1/t)$ is increasing for $t \leq e^{-1}$,  
\[ d'\alpha' \log (1/d') \le \frac{d}{10}\alpha \log(10/d) \le \frac{d \alpha \log(1/d)}{2} \le \frac{\log n}{200} \le \frac{\log (n')}{100}. \]

Thus we can use the 
induction bound $\sigma(G) \ge f(n', \alpha',d') \ge \frac{1}{50}(d')^4 (n')^{\frac{1}{2} + \frac{1}{40d'\alpha'}}$.
Let $d' = qd$, so $q \leq 1/10$. As $\alpha' \leq \alpha$ note that
\begin{align*}
 f(n',\alpha',d') &\ge \frac{1}{50}(d')^4 (n')^{\frac{1}{2} + \frac{1}{40d'\alpha'}} 
\ge \frac{1}{50} (d^4n^{\frac{1}{2} + \frac{1}{40d\alpha}}) (q^4 n^{\frac{1}{40d\alpha}(\frac{1}{q}-1)}/8) \\
&= \frac{1}{400}\big(d^4n^{\frac{1}{2} + \frac{1}{40d\alpha}}\big) \big(e^{\frac{\log n}{40d\alpha}(\frac{1}{q}-1) - 4 \log (\frac{1}{q})}\big).
\end{align*}
Since $40 d\alpha \le d\alpha \log(1/d) \le \log n/ 100$, we have that
$\frac{\log n}{40d\alpha}(\frac{1}{q}-1) - 4 \log (\frac{1}{q}) \geq 100(\frac{1}{q}-1) - 4 \log (\frac{1}{q})$
and the right-hand side of the above displayed inequality is minimized when $q$ is maximized. By using $q \leq 1/10$ we see that
\[ \sigma(G) \ge f(n', \alpha', d') \ge d^4n^{\frac{1}{2} + \frac{1}{40d\alpha}}. \]

\noindent {\bf Case 2} : $d' > d/10$.

Note that by the inequalities $d \ge n^{-1/4}$ and $n \ge 10^{80}$,
we have 
\[ (d')^2n' \ge d^2n/800 \ge n^{1/2}/800 \ge 1600. \]
Therefore we can apply Lemma \ref{DRC} to $G[V'']$ and find
a subset $V_1 \subset V''$ with
$|V_1| \ge d'n'/50 \ge dn/4000$ and the property that 
every pair of vertices in $V_1$ has at least
$10^{-9}(d')^5n' \ge 10^{-15}d^5n$ internally vertex-disjoint paths of length $4$ between them, which only use vertices in $V'' \setminus V_1$. Then by Lemma \ref{ES}, we get a subset
$U \subset V_1$ with (note that $d \le 10^{-20}$ and $d\alpha \ge 1/2$)
\[ |U| \geq \left(\frac{e}{8d}\right)^{-8d\alpha}|V_1| \ge 
\left(\frac{e}{8d}\right)^{-8d\alpha}\frac{dn}{4000} \ge e^{-15d\alpha \log(1/d)} n= d^{15d \alpha} n \]
such that the induced subgraph of $G$ with vertex set $U$ has
independence number at most $\beta := \lfloor 8d\alpha \rfloor \geq 4$. Redefine $U$ as
an arbitrary subset of size $u = \lfloor e^{-15d\alpha \log(1/d)}n \rfloor$.
Note that $u \ge  e^{-15d\alpha \log(1/d)}n/2$ since
$d\alpha \log(1/d) \le (\log n) /100$. 

Let $\rho = (d^{(6-30d\alpha)}n^{-1})^{1/(2\beta-1)}$. We have $\rho < 1$ from $d \leq 1$, $d\alpha \log(1/d) \le (\log n) /100$, and $\beta \geq 4$, since
\[ \rho^{2\beta-1} = d^{(6-30d\alpha)}n^{-1} \le d^{-30d\alpha}n^{-1} = e^{30d\alpha\log(1/d)}n^{-1} \le n^{3/10} \cdot n^{-1} < 1. \]
By applying Lemma \ref{snlem} with this value of $\rho$ 
to the graph $G[U]$, we get a subset $S$ of size
$s := \lceil \rho^{\beta-1}u \rceil $ with at most $\rho s^2$ non-adjacent
pairs. Note that if $s = 1$, then $S$ contains no non-adjacent pairs, and if $s \ge 2$, then
$S$ contains at most $\rho s^2 \le 4\rho^{2\beta-1} u^2$ non-adjacent pairs.
In any case, $S$ has at most $4\rho^{2\beta-1}u^2$ non-adjacent pairs. 
By definition, $\beta=\lfloor 8d\alpha \rfloor\geq 8d\alpha -1$ and therefore
$15d\alpha \leq 2\beta+2$.
Thus, the number of vertices in $S$ is at least
\[ s \ge \rho^{\beta - 1}u \ge (d^{(6-30d\alpha)}n^{-1})^{(\beta-1)/(2\beta-1)} \cdot(d^{15d\alpha}n/2)
= \frac{1}{2}d^{\frac{6\beta-6+15d\alpha}{2\beta-1}} n^{\frac{1}{2} + \frac{1}{2(2\beta - 1)}} \ge \frac{1}{2} d^4n^{\frac{1}{2} + \frac{1}{40d\alpha}} \]
and the number of non-adjacent pairs in $S$ is at most
\[ \rho s^2 \le 4\rho^{2\beta - 1} u^2 
\le 4(d^{6-30d\alpha} n^{-1})(d^{30d\alpha}n^2) = 4d^6 n. \]

The vertices of $S$ form the vertices of a $K_s$-subdivision,
where we use the edges between them as paths, and for the pairs in
$S$ that are not adjacent, we use paths of length $4$ between them.
We can greedily pick these paths of length $4$ as there are at most 
$4d^6n$ edges missing, and each pair of
vertices have $10^{-15}d^5n \ge 3 \cdot 4 d^6 n$ internally vertex-disjoint paths of
length $4$ between them, where we used the fact that $d \le 10^{-20}$. This completes the proof. 
\end{proof}

Note that the lemma above is no longer true if we completely remove
the restriction $\alpha \le n/2$. For example, if $\alpha = n-1$, then
we can have $d = 1/{n \choose 2}$, for which we have 
$d^4 n^{1/2 + 1/(40d\alpha)} \gg n$. The conclusion of the lemma is clearly impossible in this case since the
total number of vertices is $n$.

The proof of Theorem \ref{main1} easily follows from the two lemmas
above.

\begin{proofof}{Theorem \ref{main1}}
Let $G$ be a graph with $n$ vertices, edge density $d$,
and independence number $\alpha$. Let $c = 10^{-20}$ 
and $c_1=c_2=c'=10^{-114}$. 
If $n \le 10^{14}c^{-5} = 10^{114}$, then we have
$\sigma(G) \ge 1 \ge c'n$.  Thus we assume that $n > 10^{14}c^{-5}$. 

\noindent {\bf Case 1}: $\alpha \le 2\log n$.

If $d \ge c$, then by
the first part of Lemma \ref{dense} we have 
$\sigma(G) \ge (10^{-6}c^{5/2})n^{\alpha/(2\alpha-1)} \ge c'n^{\alpha/(2\alpha-1)}$. 

On the other hand if $d \le c$, then
we have $d\alpha \log(1/d) \le (\log n)/1000$ and $\alpha \le 2\log n \le n/2$. 
Lemma \ref{sparse} therefore implies 
\[ \sigma(G) \ge \frac{1}{50} d^4 n^{\frac{1}{2} + \frac{1}{40d\alpha}}. \]
From the inequality $d\alpha \log(1/d) \le (\log n)/1000$, we have $d^4 n^{1/(80d\alpha)} = e^{-4\log(1/d) + (\log n)/(80d\alpha)} \ge e^{-4\log(1/d) + (25/2)\log(1/d)}>1$, and thus
the above is at least
\[ \frac{n^{\frac{1}{2} + \frac{1}{80d\alpha}}}{50}  \ge \frac{n^{\frac{1}{2} + \frac{1}{\alpha}} }{50}
\ge \frac{n^{\frac{\alpha}{2\alpha-1}}}{50}  \ge c'n^{\alpha/(2\alpha-1)}. \]

\noindent {\bf Case 2}: $\alpha = a\log n$ for some $a>2$.

If $d \ge \frac{c}{a \log a}$, then by
the second part of Lemma \ref{dense} we have 
\[ \sigma(G) \ge \sqrt{\frac{c}{600}} \sqrt{\frac{n}{a\log a}} \ge c'\sqrt{\frac{n}{a\log a}}. \]

Thus we may assume that $d \le \frac{c}{a\log a}$.
If $\alpha > n/2$, then $a > n/(2\log n)$ and thus $c' \sqrt{n/(a\log a)} < 1$. 
Therefore we trivially have $\sigma(G) \ge 1 > c'\sqrt{n/(a\log a)}$ in this case.

Otherwise, $\alpha \le n/2$ and $d \le \frac{c}{a\log a}\leq c$ since $a \geq 2$. We
also have $d\alpha \log(1/d) \le (\log n)/100$. Indeed, as $t \log (1/t)$ is increasing for $t \leq e^{-1}$, to verify 
this inequality one can substitute $d =\frac{c}{a\log a}$ and $\alpha=a \log n$. By Lemma \ref{sparse},
\[ \sigma(G) \ge \frac{1}{50} d^4 n^{\frac{1}{2} + \frac{1}{40d\alpha}}. \]
Note that as in Case 1 we have $d^4 n^{1/(80d\alpha)} = e^{-4\log(1/d) + \log n/(80d\alpha)} >1$, and thus
the above is at least
\[ \frac{n^{\frac{1}{2} + \frac{1}{80d\alpha}}}{50}  \ge \frac{n^{\frac{1}{2}}}{50}  \ge c' \sqrt{\frac{n}{a \log a}}. \]
\end{proofof}

\section{Concluding Remarks} 

In this paper we established the conjecture of Erd\H{o}s and Fajtlowicz that $\frac{\chi(G)}{\sigma(G)} \leq 
\frac{Cn^{1/2}}{\log 
n}$ for every graph $G$ on $n$ vertices. The main part of the proof is Theorem \ref{main1}, which gives a lower bound on 
$f(n,\alpha)$, the minimum of $\sigma(G)$ over all graphs $G$ on $n$ vertices with $\alpha(G) \leq \alpha$. It would be 
interesting to determine the order of growth of $f(n,\alpha)$. As remarked in the introduction, determining $f(n,\alpha)$ is 
equivalent to the following Ramsey-type problem. Determine the minimum $n$ for which every red-blue edge-coloring of $K_n$ 
contains a red subdivision of $K_s$ or a blue $K_{\alpha+1}$.

Theorem \ref{main1} and the remarks afterwards determines the order of growth of $f(n,\alpha)$ for $\alpha=2$ and 
$\alpha=\Theta(\log n)$. We conjecture that the lower bound $f(n,\alpha) \geq c_1n^{\frac{\alpha}{2\alpha-1}}$ 
is tight up to a constant factor also for all $\alpha<2\log n$.
As in the case $\alpha =2$, to prove such a result
it would be sufficient to find a $K_{\alpha+1}$-free graph $G$ on $n$ vertices in which every subset of order 
$Cn^{\frac{\alpha}{2\alpha-1}}$ contains at least $n$ edges. Then the complement of $G$ will have independence number at most
$\alpha$ and no clique subdivision of order $Cn^{\frac{\alpha}{2\alpha-1}}$ (the proof is as in the case $\alpha=2$, see 
the introduction). A potential 
source of such graphs $G$ are $(n,d,\lambda)$-graphs, introduced by Alon. An $(n,d,\lambda)$-graph is a $d$-regular graph on $n$ 
vertices for which $\lambda$ is the second largest in absolute value eigenvalue of its adjacency matrix. An 
$(n,d,\lambda)$-graph $G$ which is $K_{\alpha+1}$-free with $d=\Omega\left(n^{1-\frac{1}{2\alpha-1}}\right)$ and 
$\lambda=O\left(\sqrt{d}\right)$ would satisfy the desired property by the expander mixing lemma (see, e.g., Theorem 2.11 in the 
survey \cite{KrSu}). It is worth mentioning that this would be up to a constant factor the densest 
$(n,d,\lambda)$-graph with $\lambda=O(\sqrt{d})$ which is $K_{\alpha+1}$-free (see, e.g., Theorem 4.10 of \cite{KrSu}). The 
construction of Alon \cite{Al} in the case $\alpha=2$ is the only known example of such graphs. We think it would be quite 
interesting to find examples for larger $\alpha$, which would have other applications as well.

We make the following conjecture on the order of the largest clique subdivision which one can find in a graph with chromatic 
number $k$. 
\begin{conjecture}\label{ourconjecture} There is a constant $c>0$ such that every graph $G$ with chromatic number $\chi(G)=k$ 
satisfies $\sigma(G) \geq c\sqrt{k \log k}$. \end{conjecture}

The bound in Conjecture \ref{ourconjecture} would be best possible by considering a random graph of order $O(k\log k)$. Recall 
the result of of Bollob\'as and Thomason \cite{BT} and Koml\'os and Szemer\'edi \cite{KS} which says that every graph $G$ of 
average degree $d$ satisfies $\sigma(G)=\Omega(\sqrt{d})$. This is enough to imply the bound $\sigma(G) =\Omega(\sqrt{k})$ for 
$G$ with $\chi(G)=k$, but not the extra logarithmic factor. Much of the techniques developed in this paper to solve Theorem 
\ref{main} are most useful in rather dense graphs. These techniques do not appear sufficient to solve Conjecture 
\ref{ourconjecture}, as in this conjecture one needs to handle clique subdivisions in rather sparse graphs.

\vspace{0.1cm} 
\noindent {\bf Acknowledgments.} We would like to thank Noga Alon for helpful discussions.
We would also like to thank the two anonymous referees for their valuable comments.

\end{document}